\documentclass[12pt,a4paper,reqno]{amsart}

\usepackage[T1]{fontenc}
\usepackage{amsmath,amssymb,amsthm,mathtools}
\usepackage{enumitem}
\usepackage[a4paper,
left=1in,right=1in,top=1.2in,bottom=1.2in,%
footskip=.25in]{geometry}
\usepackage[colorlinks=false,pdfborderstyle={/S/U/W 0}]{hyperref}
\usepackage{comment}

\usepackage{parskip}
\theoremstyle{plain}
\newtheorem{theorem}{Theorem}[section]
\newtheorem{lemma}[theorem]{Lemma}
\newtheorem{proposition}[theorem]{Proposition}
\newtheorem{corollary}[theorem]{Corollary}

\theoremstyle{remark}
\newtheorem{remark}[theorem]{Remark}

\newtheorem*{acknowledgements}{Acknowledgements}

\numberwithin{equation}{section}
\setlist[itemize]{leftmargin=*}

\newcommand{\cA}{\mathcal{A}}
\newcommand{\cC}{\mathcal{C}}
\newcommand{\cL}{\mathcal{L}}

\newcommand{\cR}{\mathcal{R}}
\newcommand{\cS}{\mathcal{S}}

\newcommand{\dd}{\mathrm{d}}
\newcommand{\SO}{\mathrm{SO}}
\newcommand{\R}{\mathbb{R}}

\DeclareMathOperator{\Ric}{\mathrm{Ric}}
\DeclareMathOperator{\Tr}{\mathrm{Tr}}
\DeclareMathOperator{\Hess}{\mathrm{Hess}}

\title{Cylindrical Generalized Ricci Solitons in Three Dimensions}

\author[Miguel Pino Carmona]{Miguel Pino Carmona} 
\address{Department of Mathematics, Universidad Nacional de Educación a Distancia, Spain}
\email{mpino185@alumno.uned.es}

\date{}

\begin{document}

\begin{abstract}
We construct an explicit two-parameter family of complete, non-compact, three-dimensional, smooth steady gradient generalized Ricci solitons with $\SO(2)\times\R$ symmetry, providing a cylindrical counterpart to the spherically symmetric solitons recently found by Podestà and Raffero. The family is parametrized by a flux constant $k>0$ and a conserved quantity $\cC\ge 0$. For $\cC=0$, the asymptotic geometry exhibits power-law decay; for $\cC>0$, the metric converges exponentially fast to a flat cylinder of finite radius.
\end{abstract}

\maketitle
\setcounter{tocdepth}{1} 
\tableofcontents

\section{Introduction}
\label{sec:introduction}

Let $(M, g)$ be an oriented Riemannian manifold and let $H$ be a closed 3-form on $M$. The pair $(g, H)$ is called a \emph{steady gradient generalized Ricci soliton} \cite{GarciaFernandezStreets2021,PodestaRaffero2025} if there exists a function $f \in C^\infty(M)$ such that
\begin{align}
\Ric^g + \Hess^g f - \frac{1}{2}H \circ_g H &= 0 \,, \label{eq:einsteinGSS} \\
\Delta^g H + \cL_{\nabla f} H &= 0 \,. \label{eq:maxwellGSS}
\end{align}

These are the self-similar steady gradient solutions to the \emph{generalized Ricci flow} \cite{GarciaFernandezStreets2021}, a generalization of Hamilton's classical steady gradient Ricci flow \cite{Hamilton1995}
\begin{equation*}
\Ric^g + \Hess^g f = 0    
\end{equation*}

including the coupling with the closed $3$-form $H$. By analogy with the nomenclature in the supergravity literature, we refer to Equations \eqref{eq:einsteinGSS} and \eqref{eq:maxwellGSS} as the \emph{Einstein} and \emph{Maxwell} equations, respectively.

Here, $\Hess^g=\nabla^g\dd$ is the Hessian of $g$, $\nabla^g$ the Levi-Civita connection of $g$, $\Ric^g$ is its Ricci tensor\footnote{In our conventions, the Riemann tensor is defined by
\begin{equation*}
\cR^g_{v_1,v_2} v_3 = \nabla^g_{v_1} \nabla^g_{v_2} v_3 - \nabla^g_{v_2} \nabla^g_{v_1} v_3 - \nabla^g_{[v_1,v_2]} v_3 \,, \qquad v_1, v_2, v_3\in \mathfrak{X}(M) \,.
\end{equation*}
Hence, in an orthonormal frame $\{e_1,e_2,e_3\}$, the Ricci tensor $\Ric^g$ is defined by:
\begin{equation*}
\Ric^g(v_1,v_2) = \sum_i^3 g(\cR^g_{e_i,v_1}v_2,e_i) \,, \qquad v_1,v_2 \in \mathfrak{X}(M) \,.
\end{equation*}}, $s_g$ is its scalar curvature, $\delta^g$ is the formal adjoint of the exterior derivative $\dd$, $\Delta^g=\delta^g\dd + \dd\delta^g$ is the Hodge laplacian, $\iota$ is the interior product, and $H \circ_g H$, $|H|_g^2$ are given by:
\begin{gather*}
(H\circ_g H)(v_1,v_2) = \langle\iota_{v_1}H,\iota_{v_2}H\rangle_g = \frac{1}{2}\sum_{i,j,k=1}^3 H(v_1,e_j,e_k) H(v_2,e_j, e_k) \,, \\
\vert H \vert_g^2 = \frac{1}{3}\sum_{i=1}^3 (H\circ_g H)(e_i,e_i) = \frac{1}{6}\sum_{i,j,k=1}^3 H(e_i,e_j,e_k) \, H(e_i,e_j,e_k)    
\end{gather*}

in terms of any local orthonormal frame $\{e_i\}_{i=1}^n$ and vectors $v_1,v_2 \in \mathfrak{X}(M)$, with $\langle\,,\rangle_g$ the induced inner product on $k$-forms, whose corresponding norm is $\vert\cdot\vert_g^2$.

Despite their geometric and physical importance, examples of non‑trivial steady gradient generalized Ricci solitons are scarce, let alone explicit ones. Recently, Podestà and Raffero \cite{PodestaRaffero2025} constructed the first known family of complete solitons. These solitons possess $\SO(3)$ (spherical) symmetry, hence they exhibit positive sectional curvature, and thus can be interpreted as a natural generalization of the Bryant soliton \cite{Bryant2005} within the generalized framework.

In this note, we expand this landscape and construct a two-parameter family of smooth complete steady gradient generalized Ricci solitons on $\R^3$ with $\SO(2)\times \R$ (cylindrical) symmetry. More explicitly, we fix cylindrical coordinates $(t,\theta,z)$ on $\R^3$ and set the metric
\begin{equation}\label{eq:metric}
g = dt \otimes dt + \phi(t)^2 \,d\theta \otimes d\theta + \psi(t)^2 \,dz \otimes dz \,, \qquad \phi, \psi \in C^\infty(M) \,, \quad \phi, \psi > 0 \,,
\end{equation}
the 3-form
\begin{equation}\label{eq:H}
H = h(t) \,dt \wedge d\theta \wedge dz, \quad h \in C^\infty(M) \,, 
\end{equation}
and the function $f=f(t)$, and solve Equations \eqref{eq:einsteinGSS}--\eqref{eq:maxwellGSS} for this ansatz. In particular, solving the Maxwell equation \eqref{eq:maxwellGSS} gives
\begin{equation*}
H = k e^f \sqrt{\det g} \, dt \wedge d\theta \wedge dz \,, \qquad k \in \R\setminus \{0\} \,.
\end{equation*}

Additionally, taking the divergence of \eqref{eq:einsteinGSS} shows that a soliton has an associated conserved quantity $\cC$ (see Proposition \ref{prop:sgsolitonsaresgrs}), thus the problem reduces to three unknown functions $(\phi,\psi,f)$ and two parameters $k$, $\cC$. We then prove the following explicit classification result.

\begin{theorem}\label{thm:main}
For every $k>0$, $\cC\ge0$, there exists a unique smooth complete gradient generalized Ricci soliton $(\phi, \psi, f)$ within the class of diagonal cylindrical metrics and functions, given by

\begin{equation*}
\phi(t) = \frac{t}{\sqrt{1 + (kt/2)^2}} \,, \qquad 
\psi(t) = \frac{1}{\sqrt{1 + (kt/2)^2}} \,, \qquad 
f(t) = - \ln\big[ 1 + (kt/2)^2 \big]
\end{equation*}

if $\cC=0$, and

\begin{equation*}
\begin{gathered}
\phi(t) = \frac{2}{\sqrt{\cC}}\,\frac{\tanh(\sqrt{\cC}\,t/2)}{\sqrt{1 + \frac{k^{2}}{\cC}\tanh^{2}(\sqrt{\cC}\,t/2)}} \,, \qquad \psi(t) = \frac{1}{\sqrt{1 + \frac{k^{2}}{\cC}\tanh^{2}(\sqrt{\cC}\,t/2)}} \, \\[6pt]
f(t) = \ln\left[\frac{1}{2}\left(1-\frac{\cC}{k^{2}}\right) + \frac{1}{2}\left(1+\frac{\cC}{k^{2}}\right)\frac{1 - \frac{k^{2}}{\cC}\tanh^{2}(\sqrt{\cC}\,t/2)}{1 + \frac{k^{2}}{\cC}\tanh^{2}(\sqrt{\cC}\,t/2)}
\right] \,.
\end{gathered}
\end{equation*}

if $\cC>0$. In particular, the $\cC>0$ case converges to the $\cC=0$ case when $\cC\to 0$.
\end{theorem}

In particular, the geometry of the solutions exhibits different behaviors with the value of $\cC$. Define the area of the unit height cylinder by $\cS = 2\pi \phi\psi$, i.e. the angular integral of the square root of the determinant of the metric.

\begin{corollary}\label{cor:cuspend}
The asymptotic geometry of the diagonal cylindrical steady gradient generalized Ricci solitons of Theorem \ref{thm:main} depends on the dilatonic constant $\cC$.
\begin{enumerate}[label=(\roman*), leftmargin=*]
\item If $\cC=0$, the soliton exhibits power‑law decay: the area density of a principal orbit decays as $8\pi/(k^{2}t)$, while the scalar curvature decays as $-4/t^2$.
\item If $\cC>0$, the metric converges exponentially fast to a flat cylinder of finite radius; the area density tends to $4\pi\sqrt{\cC}/(\cC+k^{2})$, and all sectional curvatures (hence the scalar curvature) decay exponentially.
\end{enumerate}
\end{corollary}

In particular, solitons with $\cC=0$ are solutions to the supergravity system \cite{Callanetal1985}. Thus, the fruitful results of Podesta's and Raffero's spherically symmetric ansatz and the cylindrical ansatz exposed here encourage to use a similar approach within the framework of \emph{Heterotic supergravity}. The Heterotic supergravity system is a generalization of the supergravity system that includes quadratic terms in curvature \cite{Moroianu2022,Moroianu2024}. In three dimensions, it admits a particularly amenable form \cite{Moroianu2024,Moroianu2026a,Moroianu2026b}.   


\section{Preliminaries}
\label{sec:preliminaries}

We start by characterizing some identities of the steady gradient soliton system. 

\begin{proposition}
\label{prop:sgsolitonsaresgrs}
A steady gradient generalized Ricci soliton satisfies
\begin{align}
\delta^gH + \iota_{\nabla f} H = \beta \,, \label{eq:maxwellbeta} \\
\Delta^g f + |\nabla f|_g^2 - |H|_g^2 = \cC \label{eq:dilatonC} \,.
\end{align}
for $\beta$ a closed 2-form and $\cC$ a constant.
\end{proposition}

\begin{proof}
First we prove \eqref{eq:maxwellbeta}. By definition of the Hodge laplacian and using Cartan's formula, it is straightforward to see that \eqref{eq:maxwellGSS} can be written as
\begin{equation*}
\dd \left(\delta^gH + \iota_{\nabla f} H\right) = 0 \,,
\end{equation*}
thus
\begin{equation*}
\delta^gH + \iota_{\nabla f} H = \beta \,, \qquad \beta \in \Omega^2_{\mathrm{cl}}(M) \,.
\end{equation*}

Now we prove \eqref{eq:dilatonC}. The Weitzenböck formula applied to $\dd f$ gives\footnote{Recall that, in our conventions, $\Delta^g f=-\Tr_g\Hess^g f$. Similarly, note the normalization in $H\circ_g H$, explained in Section \ref{sec:introduction}.}
\begin{equation*}
\dd\Delta^g f = \nabla^{g \ast}\Hess^g f + \Ric^g(\nabla f) \,.
\end{equation*}

On the other hand, the contracted Bianchi identity yields
\begin{equation*}
\nabla^{g \ast}\Ric^g = - \frac{1}{2}\dd s_g \,. 
\end{equation*}

Then, taking the divergence of \eqref{eq:einsteinGSS} and using \eqref{eq:maxwellGSS} we obtain:
\begin{align*}
0 &= \nabla^{g \ast}\left(\Ric^g + \Hess^g f - \frac{1}{2}H \circ_g H\right)(v) \\
&= \nabla^{g \ast}\Ric^g(v) + \nabla^{g \ast}\Hess^g f - \frac{1}{2}\langle\delta^gH, H(v)\rangle_g + \frac{1}{2}\langle H(e_i),\nabla^g_vH(e_i) \rangle_g \\
&= -\frac{1}{2}\dd s_g(v) + \dd\Delta^g f - \Ric^g(\nabla f) + \frac{1}{2}H \circ_g H(\nabla f,v) + \frac{1}{4}\dd|H|_g^2(v) \,.
\end{align*}

Using \eqref{eq:einsteinGSS} and its trace reduces the expression above to
\begin{align*}
0 &= -\frac{1}{2}\dd s_g(v) + \dd\Delta^g f + \Hess^g f(\nabla f) + \frac{1}{4}\dd|H|_g^2(v) \\
&= - \frac{1}{2}\dd\left(s_g -  2\Delta^g f - |\nabla f|_g^2 - \frac{1}{2}|H|_g^2 \right)(v) \\
&= \frac{1}{2}\dd\left(\Delta^g f + |\nabla f|_g^2 - |H|_g^2 \right)(v) \,,
\end{align*}

hence
\begin{equation}\label{eq:dilaton}
\Delta^g f + |\nabla f|_g^2 - |H|_g^2 = \cC \,, \qquad \cC \in \R \,.
\end{equation}
\end{proof}

Again, inspired by the supergravity literature, we refer to \eqref{eq:dilatonC} as the \emph{dilaton} equation, and to $\cC$ as the \emph{dilatonic constant}.

Henceforth, we only consider cylindrical, i.e., $\SO(2) \times \R$-invariant  solitons $(g, H, f)$ on $\R^3$ of the form\footnote{The most general smooth $\SO(2) \times \R$-invariant metric contains an off-diagonal term:
\begin{equation*}
g = dt \otimes dt + \phi(t)^2 d\theta \otimes d\theta + \zeta(t)\,(d\theta \otimes dz + dz \otimes d\theta) + \psi(t)^2 dz \otimes dz,
\end{equation*}
where $\phi, \psi > 0$ and $\zeta$ are smooth functions. This term cannot be eliminated globally by a change of coordinates while preserving the diagonal form of the metric: it corresponds to a twist in the $\SO(2)$ factor, making the principal orbits nontrivially fibered over the $t$-line. The analysis of the full system with $\zeta \neq 0$ is more involved and likely lacks explicit solutions, so here we focus on the diagonal case.}
\begin{equation}\label{eq:ansatz}
g = dt \otimes dt + \phi(t)^2 \,d\theta \otimes d\theta + \psi(t)^2 \,dz \otimes dz \,, \qquad H = h(t) \,dt \wedge d\theta \wedge dz \,, \qquad f = f(t) \,,
\end{equation}
where $\phi,\psi>0$, $h$, and $f$ are smooth functions.

To ensure that our ansatz extends through the \emph{singular orbit} (the $z$-axis) covering the whole $\R^3$, we must ensure certain compatibility conditions on the functions.

\begin{lemma}\label{lemma:smoothnessconditions}
Let $(g,H,f)$ be of the form \eqref{eq:ansatz}. Then $(g,H,f)$ extends smoothly to a neighborhood of the $z$-axis ($t=0$) if and only if
\begin{enumerate}[label=(\roman*), leftmargin=*, topsep=5pt, itemsep=5pt]
\item $\phi$ is odd, $\psi$ is even, with $\phi(0)=0$, $\phi'(0)=1$, and $\psi(0)>0$.
\item $h$ is odd with $h(0)=0$.
\item $f$ is even with $f'(0)=0$.
\end{enumerate}
\end{lemma}

\begin{proof}
The conditions follow from fixing Cartesian coordinates $(x,y,z)$ defined by $x=t\cos\theta$, $y=t\sin\theta$, $z=z$ and studying the limit at $x=0, y=0$.

By \cite[Proposition 1.4.7]{Petersen2016}, $g$ is smooth at $t=0$ if and only if $(i)$ is satisfied.

In cylindrical coordinates, $H$ takes the form
\begin{equation*}
H = \frac{h(t)}{t}\,dx\wedge dy\wedge dz \,.
\end{equation*} 

For $H$ to be smooth at $t=0$, $h(t)/t$ must be an even function of $t$, hence $h(t)$ is odd and $h(0)=0$, giving condition (ii).

Similarly, since $f(t)=f(\sqrt{x^2+y^2})$, $f$ is smooth at $t=0$ if and only if $f$ is even (hence $f'(0)=0$), proving (iii).
\end{proof}

A \emph{smooth soliton} (henceforth, simply soliton) is then a triple $(g,H,f)$ satisfying the conditions in Lemma \ref{lemma:smoothnessconditions}. 


\section{Reduction of the gradient generalized Ricci system}

We now reduce the gradient generalized Ricci soliton system \eqref{eq:einsteinGSS}--\eqref{eq:maxwellGSS} under the diagonal cylindrical ansatz \eqref{eq:ansatz}, which allows to solve the Maxwell equation \eqref{eq:maxwellGSS} explicitly. Then, we study the symmetries of the system.

\subsection{Reduction of the Maxwell equation}
\label{sec:reductionmaxwell}

Using
\begin{equation*}
H = h(t)\,dt\wedge d\theta\wedge dz \,, \qquad \ast_g H = \frac{h(t)}{\sqrt{\det g(t)}} \,, \qquad \det g = \phi^2\psi^2 \,, \qquad \nabla f = f'(t)\partial_t \,,
\end{equation*}
where $\ast_g$ is the Hodge dual, we find
\begin{equation*}
\delta^g H = -\ast_g \,\dd \,(\ast_g H) 
= -\ast_g \,\left( \left(\frac{h}{\sqrt{\det g}}\right)' dt \right) = -\left(\frac{h}{\sqrt{\det g}}\right)' \sqrt{\det g} \,d\theta\wedge dz \,.
\end{equation*}

and
\begin{equation*}
\iota_{\nabla f} H = h(t) \,dt\wedge d\theta\wedge dz(f'(t)\partial_t) = h(t) f'(t) \,d\theta\wedge dz \,.
\end{equation*}

Therefore, by Proposition \ref{prop:sgsolitonsaresgrs}:
\begin{equation}\label{eq:maxred}
\beta = \delta^g H + \iota_{\nabla f} H = 
\left[-\left(\frac{h}{\sqrt{\det g}}\right)' \sqrt{\det g} + h f' \right] d\theta\wedge dz \,.
\end{equation}

We now show that a smooth soliton satisfies $\beta=0$.

\begin{lemma}\label{lemma:maxwell}
A smooth diagonal cylindrical solution to \eqref{eq:maxwellGSS} satisfies
\begin{equation*}
\delta^gH + \iota_{\nabla f} H = 0 \,.
\end{equation*}
\end{lemma}

\begin{proof}
Since $\beta$ is closed, under the ansatz \eqref{eq:ansatz}, Equation \eqref{eq:maxwellGSS} reduces to
\begin{equation*}
\dd\left( -\left(\frac{h}{\sqrt{\det g}}\right)' \sqrt{\det g} + h f' \right) \, d\theta\wedge dz = 0 \,.
\end{equation*}

The expression inside the parentheses depends only on $t$, hence it is constant, say $\rho\in\R$:
\begin{equation}\label{eq:integmaxwellGSS}
b(t) := - \left(\frac{h(t)}{\sqrt{\det g(t)}}\right)' \sqrt{\det g(t)} + h(t) f'(t) = \rho \,.
\end{equation}

We prove $\rho=0$ by showing that $\lim_{t\to0}b(t)=0$.

By Lemma \ref{lemma:smoothnessconditions} $\phi$ is odd with $\phi'(0)=1$, $\psi$ is even with $\psi(0):=\psi_0>0$, and $f$ is even. We now study the two terms in $b$. First, since $f$ and $h$ are smooth, $f'$ is smooth and $hf'$ is smooth. In particular, $\lim_{t\to0}hf'=0$.

Now, recall that $\sqrt{\det g(t)} = \phi(t)\psi(t)$ is smooth and odd, with $\sqrt{\det g(t)}(0)=0$ and $\sqrt{\det g(t)}'(0)=\psi_0>0$ Using Hadamard's lemma, we write
\begin{equation*}
h(t) = t \,\tilde{h}(t) \,, \qquad \sqrt{\det g(t)} = t \,\tilde{g}(t) \,, \qquad \tilde{h}, \,\tilde{g} \in C^\infty(\R) \,, \qquad \tilde{h}, \,\tilde{g} \quad \text{even} \,.
\end{equation*}

Moreover,
\begin{equation*}
\tilde{g}(0) = \lim_{t\to0}\frac{\sqrt{\det g(t)}}{t} = (\sqrt{\det g(t)})' = \psi_0 > 0 \,,
\end{equation*}

thus $\left(h/\sqrt{\det g}\right)'$ is smooth at $t=0$. Then,
\begin{equation*}
\lim_{t\to0}\left(\frac{h}{\sqrt{\det g}}\right)'\sqrt{\det g} = 0 \,,
\end{equation*}

whence $\lim_{t\to0}b(t)=0$. Therefore, $\rho=0$.
\end{proof}

Thus, by Lemma \ref{lemma:maxwell}, Equation \eqref{eq:maxred} simplifies to
\begin{equation*}
\left(\frac{h}{\sqrt{\det g}}\right)' = \frac{h}{\sqrt{\det g}}\, f' \,.
\end{equation*}

Setting $u(t) = h(t)/\sqrt{\det g(t)}$, the equation above becomes $u' = u f'$, and integrating we find
\begin{equation*}
u(t) = k e^{f(t)} \,, \qquad k \in \R \,.
\end{equation*}

Returning to $h$, we obtain:
\begin{equation}\label{eq:Hsolution}
h(t) = k e^{f(t)} \sqrt{\det g(t)} = k e^{f(t)} \phi(t)\psi(t) \,.
\end{equation}

In particular, this allows to rewrite the triple $(g,H,f)$ as a tuple $(\phi,\psi,f,k)$ where $k$ is the integration constant in \eqref{eq:Hsolution}.


\subsection{Reduction of the Einstein equation}
\label{sec:reduction}

We first obtain the curvature of the ansatz metric. A direct computation shows that the unique non-vanishing Christoffel symbols of \eqref{eq:metric} are  
\begin{equation}\label{eq:christoffel}
\Gamma^{\partial_t}_{\partial_\theta\partial_\theta} = -\phi\phi' \,, \qquad \Gamma^{\partial_t}_{\partial_z\partial_z} = -\psi\psi' \,, \qquad \Gamma^{\partial_\theta}_{\partial_t\partial_\theta} = \Gamma^{\partial_\theta}_{\partial_\theta \partial_t} = \frac{\phi'}{\phi} \,, \qquad \Gamma^{\partial_z}_{\partial_t\partial_z} = \Gamma^{\partial_z}_{\partial_z\partial_t} = \frac{\psi'}{\psi} \,.
\end{equation}

Then, using the standard formula
\begin{equation*}
\Ric^g(e_i,e_j) = \sum_{k,l=1}^3 \left[\partial_{e_k}\Gamma_{e_ie_j}^{e_k} - \partial_{e_i}\Gamma _{e_pe_j}^{e_p} + \left(\Gamma _{e_ie_j}^{e_l}\Gamma_{e_ke_l}^{e_k} - \Gamma_{e_ke_j}^{e_l}\Gamma_{e_ie_l}^{e_k}\right)\right] \,,
\end{equation*}

we calculate the Ricci tensor (see also \cite[\S 4.2.4]{Petersen2016}):
\begin{equation}\label{eq:riccidoublywarped}
\Ric^g = -\left(\frac{\phi''}{\phi} + \frac{\psi''}{\psi}\right) \,dt \otimes dt -\phi\left(\phi'' + \frac{\phi'\psi'}{\psi}\right) \,d\theta \otimes d\theta -\psi\left(\psi'' + \frac{\phi'\psi'}{\phi}\right) \,dz \otimes dz \,.
\end{equation}  

Using the symmetry of $f$, its Hessian reduces to
\begin{equation*}
\Hess^g f = \nabla^g \dd f = \nabla^g (f'(t)\,dt) = f'' \,dt \otimes dt + f' \nabla^g dt \,,
\end{equation*}

whence (using \eqref{eq:christoffel}; see again \cite[\S 4.2.4]{Petersen2016})
\begin{equation}\label{eq:Hessianf}
\Hess^g f = f'' \,dt \otimes dt + \phi\phi' f' \,d\theta \otimes d\theta + \psi\psi' f' \,dz \otimes dz \,.
\end{equation}

Finally, since $\iota_vX = \iota_X \ast_g \frac{h}{\sqrt{\det g}} = \frac{h}{\sqrt{\det g}} \ast_g X^{\flat_g}$, where $\flat_g$ the musical isomorphism, we compute
\begin{equation*}
H \circ_g H(v_1,v_2) = \langle \iota_{v_1}H, \iota_{v_2}H \rangle_g = \frac{h^2}{\det g} \langle \ast_g v_1^{\flat_g}, \ast_g v_2^{\flat_g} \rangle_g \,.
\end{equation*}

Then, using the explicit solution of $h$ in \eqref{eq:Hsolution}, we obtain:
\begin{equation}\label{eq:Hsquared}
H \circ_g H = k^2 e^{2f} \,g \,.
\end{equation}

Substituting the expressions \eqref{eq:riccidoublywarped}--\eqref{eq:Hsquared} into the Einstein equation yields three independent ordinary differential equations:
\begin{align*}
-\frac{\phi''}{\phi} - \frac{\psi''}{\psi} + f'' - \frac{1}{2}k^2 e^{2f} &= 0 \,, \\
-\phi\phi'' - \phi\phi' \frac{\psi'}{\psi} + \phi\phi' f' - \frac{1}{2}k^2 e^{2f} \phi^2 &= 0 \,, \\
-\psi\psi'' - \psi\psi' \frac{\phi'}{\phi} + \psi\psi' f' - \frac{1}{2}k^2 e^{2f} \psi^2 &= 0 \,.
\end{align*}

Rearranging the system above and substituting the second and third equations into the first one, we obtain the following evolution system for the second derivatives:
\begin{align}
\phi'' &= -\frac{\psi'}{\psi} \phi' + f' \phi' - \frac{1}{2}k^2 e^{2f} \phi \,, \label{eq:phi''} \\
\psi'' &= -\frac{\phi'}{\phi} \psi' + f' \psi' - \frac{1}{2}k^2 e^{2f} \psi \,, \label{eq:psi''} \\
f'' &= \left(\frac{\phi'}{\phi} + \frac{\psi'}{\psi}\right) f' - 2\frac{\phi'\psi'}{\phi\psi} - \frac{1}{2}k^2 e^{2f} \,. \label{eq:f''}
\end{align}


\subsection{Reduction of the dilaton equation}
\label{sec:reductionsoliton}

Computing the relevant quantities
\begin{gather*}
|H|_g^2 = \frac{h^2}{\det g} = k^2 e^{2f} \,, \qquad |\nabla f|_g^2 = (f')^2 \,, \\
\Delta^g f = - f'' - \frac{(\det g)'}{2\det g} f' = - f'' - \left( \frac{\phi'}{\phi} + \frac{\psi'}{\psi} \right) f' \,,
\end{gather*}

and substituting into the dilaton equation \eqref{eq:dilaton}, we find:
\begin{equation}\label{eq:conservationlawreduced}
- f'' - \left( \frac{\phi'}{\phi} + \frac{\psi'}{\psi} \right) f' + (f')^2 - k^2 e^{2f} = \cC \,.
\end{equation}


\subsection{Normalization and symmetry reduction}
\label{sec:normalization}

The soliton equations possess several symmetries. The following transformations map solutions to solutions and preserve the diagonal cylindrical ansatz \eqref{eq:ansatz}:

\begin{enumerate}[label=(\roman*), leftmargin=*, topsep=5pt, itemsep=5pt]
\item \emph{Homothety.} For any $\lambda > 0$, the triple
\begin{equation*}
(\tilde{g}, \tilde{f}, \tilde{k}) = (\lambda^2 g, f - \ln\lambda, k)
\end{equation*}
satisfies the gradient generalized Ricci system whenever $(g,H,f)$ does. This rescaling freedom corresponds to an overall choice of length scale.

\item \emph{Shift of the potential.} The evolution system is invariant under
\begin{equation*}
(f, k) \longmapsto (f + a,\ k e^{-a}) \,, \qquad a \in \R \,,
\end{equation*}

because the combination $k e^{f}$, and consequently $k^2 e^{2f}$, remains unchanged. This reflects the fact that $f$ is defined only up to an additive constant once $k$ is fixed.

\item \emph{Orientation reversal of $H$.} Replacing $H$ by $-H$ sends $k$ to $-k$. Since only $k^2$ appears in the metric equations, we may always assume $k \geq 0$. The case $k=0$ reduces to the classical Ricci soliton system; we focus on $k>0$.
\end{enumerate}

The two continuous symmetries (homothety and $f$-shift) allow us to normalize two of the three initial parameters $(\psi_0, f_0, k)$. We proceed in two steps:

\begin{enumerate}[leftmargin=*, topsep=5pt, itemsep=5pt]
\item Apply a homothety with $\lambda = 1/\psi_0$. This yields
\begin{equation*}
\tilde{\psi}_0 = 1, \qquad \tilde{f}_0 = f_0 + \ln\psi_0, \qquad \tilde{k} = k \,.
\end{equation*}

\item Apply a $f$-shift with $a = - (f_0 + \ln\psi_0)$. This gives
\begin{equation*}
\hat{f}_0 = 0, \qquad \hat{k} = \tilde{k} e^{f_0 + \ln\psi_0} = k \psi_0e^{f_0} \,.
\end{equation*}
\end{enumerate}

After these transformations, we obtain a solution with
\begin{equation}\label{eq:normalizedinitial}
\hat{\psi}(0)=1, \qquad \hat{f}(0)=0, \qquad \hat{k} = k \psi_0e^{f_0} > 0 \,.
\end{equation}

Notice that, after the homothety, the metric reads:
\begin{equation*}
\tilde{g} = \frac{1}{\psi_0^2}\big( dt \otimes dt + \phi(t)^2\,d\theta \otimes d\theta + \psi(t)^2\,dz \otimes dz \big) \,.
\end{equation*}

To restore the condition $g(\partial_t,\partial_t)=1$, we define a rescaled radial coordinate $\hat{t} = t/\psi_0$. In this coordinate the metric becomes
\begin{equation*}
\tilde{g} = d\hat{t} \otimes d\hat{t} + \left(\frac{\phi(\psi_0\hat{t})}{\psi_0}\right)^2 d\theta \otimes d\theta + \left(\frac{\psi(\psi_0\hat{t})}{\psi_0}\right)^2 dz \otimes dz \,.
\end{equation*}

Introducing the new warping functions
$\hat{\phi}(\hat{t}) = \frac{\phi(\psi_0\hat{t})}{\psi_0}$,
$\hat\psi(\hat{t}) = \frac{\psi(\psi_0\hat{t})}{\psi_0}$
and evaluating at $\hat{t} = 0$ yields
$\hat{\psi}(0)=1$, $\hat{\psi}'(0)=0$, $\hat{\phi}(0)=0$, $\hat{\phi}'(0)=1$.

On the other hand, since $\hat{k}$ is simply a positive rescaling of the original $k$, we may safely rename $\hat{k}$ as $k$ in what follows, and analogously we omit the hat superscript for the rest of variables. Hence every solution of the reduced system is equivalent, modulo the above symmetries, to a unique solution satisfying the normalized initial conditions
\begin{equation}\label{eq:boundaryconditions}
\phi(0)=0 \,, \quad \phi'(0)=1 \,, \quad \psi(0)=1 \,, \quad \psi'(0)=0 \,, \quad f(0)=0 \,, \quad f'(0)=0
\end{equation}

that ensure smoothness at the $t=0$ axis. This normalization will be used throughout the remainder of the paper.

However, once we fix $k$, not every solution satisfying the initial conditions \eqref{eq:boundaryconditions} is uniquely determined: since the system is singular at $t=0$, we rather have a family of solutions. We characterize this freedom using the reduced dilaton equation \eqref{eq:conservationlawreduced}, that is, the reduced conservation law associated to the system, through the constant $\cC$. Thus, for each fixed $k>0$, the initial conditions \eqref{eq:boundaryconditions} admit a one‑parameter family of smooth solutions, distinguished by the value of $\cC$; altogether we obtain a two‑parameter family of normalized solitons labeled by $k>0$ and $\cC$. As studied in Section \ref{sec:gradient generalized Ricci}, these solutions will also be complete when $\cC\ge0$.


\section{Explicit integration and global properties}
\label{sec:gradient generalized Ricci}

We now analyze the global structure of the gradient generalized Ricci solitons with diagonal cylindrical symmetry. Throughout this section, we work with the normalization \eqref{eq:normalizedinitial}.


\subsection{Calculation of explicit solutions}

For a diagonal cylindrical soliton $(g,H,f)$, we define the \emph{orbit area function}:
\begin{equation}\label{eq:areafunction}
\cA(t) = \sqrt{\det g(t)} = \phi(t)\psi(t) \,, \qquad \cA(0) = 0 \,, \qquad \dot{\cA}(0) = 1 \,.
\end{equation}

Similarly, we define the function:
\begin{equation}\label{eq:mu}
\mu(t) := \frac{\phi(t)}{\psi(t)} \,, \qquad \mu(0) = 0 \,, \qquad \dot{\mu}(0) = 1 \,.
\end{equation}

Then 
\begin{equation}\label{eq:phipsimutau}
\phi = \sqrt{\cA\mu} \,, \qquad \psi = \sqrt{\cA/\mu} \,.
\end{equation}

For convenience, we introduce the variable
\begin{equation}\label{eq:tau}
\tau(t) = \int_0^{t}e^{f(s)} \,ds, \qquad t \ge 0 \,.
\end{equation}

Because $f$ is smooth, $e^{f(s)}>0$ and $\tau$ is a strictly increasing smooth function of $t$, hence it is a diffeomorphism onto its image $[0,\tau_{\max})$, where $\tau_{\max} = \lim_{t\to\infty}\tau(t) \le \infty$. Denote the inverse function by $t(\tau)$. Then $dt/d\tau=e^{-f(t(\tau))}$, hence
\begin{equation}\label{eq:tautderivatives}
\frac{d}{d\tau}=e^{-f}\frac{d}{dt} \,.
\end{equation}

Often, we use $\cdot$ to denote the derivative with respect to $\tau$.

First, we obtain the explicit solution of $V$ in terms of $\tau$.

\begin{lemma}\label{lemma:Atau}
Let $(\phi,\psi,f,k)$ be a gradient generalized Ricci soliton normalized as in \eqref{eq:normalizedinitial}. Then, in terms of $\tau$, $\cA$ is given by
\begin{equation}\label{eq:Atauexplicit}
\cA(\tau) = \frac{1}{k}\sin(k\tau) \,,
\end{equation}
with $0\le\tau<\pi/k$.
\end{lemma}

\begin{proof}
Differentiating \eqref{eq:areafunction} twice and using the Einstein equations \eqref{eq:phi''}--\eqref{eq:psi''}, we find
\begin{align*}
\cA'' &= \phi''\psi + 2\phi'\psi' + \phi\psi'' \\
&= \left(-\frac{\psi'}{\psi} \phi' + f' \phi' - \frac{1}{2}k^2 e^{2f} \phi\right)\psi + 2\phi'\psi' + \phi\left(-\frac{\phi'}{\phi} \psi' + f' \psi' - \frac{1}{2}k^2 e^{2f} \psi\right) \\
&= f' (\phi'\psi + \phi\psi') - k^2 e^{2f} \phi\psi = f'\cA' - k^2 e^{2f} \cA \,,
\end{align*}

or, equivalently,
\begin{equation}\label{eq:Vequation}
\cA'' - f'\cA' + k^2 e^{2f} \cA = 0 \,.
\end{equation}

Multiplying \eqref{eq:Vequation} by the integrating factor $e^{-f}$ gives, after rearranging,
\begin{equation}\label{eq:Vequationintegfactor}
e^{-f}(e^{-f}\cA')' + k^2 \cA = 0 \,.
\end{equation}

In terms of $\tau$ and using \eqref{eq:tautderivatives}, Equation \eqref{eq:Vequationintegfactor} becomes the harmonic oscillator equation
\begin{equation}\label{eq:Atauode}
\ddot{\cA} + k^2 \cA = 0 \,, \qquad \cA(0) = 0 \,, \qquad \dot{\cA}(0) = 1 \,.
\end{equation}

The ODE \eqref{eq:Atauode} has solution
\begin{equation*}
\cA(\tau) = \frac{1}{k}\sin(k\tau) \,, \qquad 0 \le \tau < \tau_{\max} \,.
\end{equation*}

Since $\cA(t)>0$ for $t>0$, we must have $0<k\tau<\pi$, thus $\tau_{\max}<\pi/k$.
\end{proof}

Second, we obtain the expression for $\mu$ in terms of $\tau$.

\begin{lemma}\label{lemma:mutau}
Let $(\phi,\psi,f,k)$ be a gradient generalized Ricci soliton normalized as in \eqref{eq:normalizedinitial}. Then, in terms of $\tau$, $\mu$ is given by
\begin{equation}\label{eq:solutionmu}
\mu(\tau) = \frac{2}{k}\,\tan(k\tau/2) \,,
\end{equation}
with $0\le\tau<\pi/k$.
\end{lemma}

\begin{proof}
Differentiating \eqref{eq:mu} with respect to $\tau$ gives
\begin{equation}\label{eq:mudprime}
\dot{\mu} = \frac{\dot{\phi}\psi - \phi\dot{\psi}}{\psi^{2}} \,, \qquad \ddot{\mu} = \frac{\ddot{\phi}\psi - \phi\ddot{\psi}}{\psi^{2}} - 2\frac{\dot{\psi}}{\psi}\,\dot{\mu} \,.
\end{equation}

After a straightforward computation using $dt/d\tau=e^{-f}$, the Einstein component equations \eqref{eq:phi''}--\eqref{eq:psi''} become
\begin{equation*}
\ddot{\phi} = -\frac{\dot{\psi}}{\psi}\,\dot{\phi} - \frac{1}{2} k^{2}\phi, \qquad
\ddot{\psi} = -\frac{\dot{\phi}}{\phi}\,\dot{\psi} - \frac{1}{2} k^{2}\psi \,.
\end{equation*}

In particular, $\ddot{\phi}\psi - \phi\ddot{\psi} = 0$, hence \eqref{eq:mudprime} reduces to $\ddot{\mu} = - 2\dot{\psi}/\psi\,\dot{\mu}$, or, equivalently,
\begin{equation}\label{eq:mudprime2}
\frac{\ddot{\mu}}{\dot{\mu}} = -2\frac{\dot{\psi}}{\psi} \,.
\end{equation}

Now, observe that 
\begin{equation*}
\frac{\dot{\mu}}{\mu} = \frac{\dot{\phi}}{\phi} - \frac{\dot{\psi}}{\psi} \,, \qquad \frac{\dot{\cA}}{\cA} = \frac{\dot{\phi}}{\phi} + \frac{\dot{\psi}}{\psi} \,,
\end{equation*}
whence 
\begin{equation*}
2\frac{\dot{\psi}}{\psi} = \frac{\dot{\cA}}{\cA} - \frac{\dot{\mu}}{\mu} \,.
\end{equation*}

Therefore, we can write \eqref{eq:mudprime2} as
\begin{equation}\label{eq:muode}
\frac{d}{d\tau}\ln\left(\frac{\dot{\mu}}{\mu}\right) = -\frac{d}{d\tau}\ln \cA \,.
\end{equation}

Integrating \eqref{eq:muode} we obtain $\dot{\mu} = \mu/\cA$, where we have used that $\lim_{\tau\to0}\mu/\cA=1$. Equivalently,
\begin{equation*}
\frac{d}{d\tau} \ln \mu = \frac{1}{\cA} = \frac{k}{\sin(k\tau)} \,.
\end{equation*}

Integrating once again, we obtain
\begin{equation*}
\mu(\tau) = C \,\tan(k\tau/2) \,,
\end{equation*}

where $C$ is a constant. Differentiating the expression above, we obtain
\begin{equation*}
\dot{\mu}(\tau) = \frac{Ck}{2}\,\cos^{-2}\left(\frac{k\tau}{2}\right) \,,
\end{equation*}
and, using $\dot{\mu}(0)=1$, we obtain $C=2/k$, whence
\begin{equation*}
\mu(\tau) = \frac{2}{k}\,\tan(k\tau/2) \,.
\end{equation*}

Since $\mu(t)>0$ for $t>0$, we must have $0<k\tau<\pi$, thus $\tau_{\max}<\pi/k$.
\end{proof}

Before obtaining the expression for $f$ in terms of $\tau$, we show that solitons with a negative value of the $\cC$ constant have finite diameter in the $t$-direction, hence cannot be complete solutions, and must be excluded from the analysis

\begin{proposition}\label{prop:dilatonconstant}
Let $(\phi,\psi,f,k)$ be a gradient generalized Ricci soliton normalized as in \eqref{eq:normalizedinitial}. Then the solution has infinite diameter in the $t$-direction if and only if $\cC \ge 0$. In particular, for $\cC < 0$ the manifold has finite diameter in the $t$-direction and is incomplete.
\end{proposition}

\begin{proof}
From the Einstein equation $\eqref{eq:einsteinGSS}$ and the Maxwell equation $\eqref{eq:maxwellGSS}$, the metric components $\phi,\psi$ and the $H$-flux satisfy the reduced Einstein equation \eqref{eq:phi''}–\eqref{eq:f''}. As shown in Section \ref{sec:gradient generalized Ricci}, \eqref{eq:phi''} and \eqref{eq:psi''}, determine $\phi$ and $\psi$ uniquely in terms of the variable $\tau(t) = \int_0^t e^{f(s)} \,ds$, independently of the dilaton equation. Explicitly,
\begin{equation*}
\phi(\tau) = \frac{2}{k}\sin(k\tau/2),\qquad \psi(\tau) = \cos(k\tau/2),
\end{equation*}
and consequently $\cA(\tau) = \phi\psi = \frac{1}{k}\sin(k\tau)$. The function $\cA(\tau)$ is positive precisely for $0 < k\tau < \pi$, so $\tau$ can never exceed $\pi/k$; the maximal possible interval is $[0,\pi/k)$. The actual domain of $\tau$ will be $[0,\tau_{\max})$, where $\tau_{\max} \le \pi/k$ is determined by the condition $y(\tau) = e^{f(\tau)} > 0$.

Recall that $\cA'/\cA=\phi'/\phi+\psi'/\psi$. Then, using $d/d\tau=e^{-f}\,d/dt$, we compute
\begin{equation*}
f' = e^f\dot{f} \,, \qquad f'' = e^{2f}(\ddot f + \dot{f}^2) \,, \qquad \frac{\cA'}{\cA} = e^f \,\frac{\dot{\cA}}{\cA} = e^f k\cot(k\tau) \,.
\end{equation*}

Thus, we can write \eqref{eq:f''} as
\begin{equation*}
0 = f'' - \left(\frac{\phi'}{\phi} + \frac{\psi'}{\psi}\right) f' + 2\frac{\phi'\psi'}{\phi\psi} + \frac{1}{2}k^2 e^{2f} = e^{2f}(\ddot{f} + (\dot{f})^2) - e^{2f} k\cot(k\tau) \,\dot{f} \,,
\end{equation*}
hence
\begin{equation*}
0 = e^f(\ddot{f} + (\dot{f})^2) - e^f k\cot(k\tau) \,\dot{f} = \frac{d^2}{d\tau^2}(e^f) + k\cot(k\tau) \,\frac{d}{d\tau}(e^f) \,.
\end{equation*}

Introducing the variable $y=e^f$, we obtain the linear second-order ODE
\begin{equation*}
\ddot{y} - k\cot(k\tau)\,\dot{y} = 0 \,,
\end{equation*}
with initial conditions $y(0)=1,\ \dot{y}(0)=0$. The general solution is
\begin{equation}\label{eq:ysolution}
y(\tau) = A + D\cos(k\tau) \,, \qquad A+D=1 \,.
\end{equation}

Now, expressing \eqref{eq:conservationlawreduced} in terms of $\tau$ yields
\begin{equation*}
- e^{2f}\left(\ddot{f} + k\cot(k\tau)\dot{f} + k^2\right) = \cC \,,
\end{equation*}
and, using $y=e^f$ and
\begin{equation*}
f = \ln y \,, \qquad \dot{f} = \frac{\dot{y}}{y} \,, \qquad \ddot{f} = \frac{\ddot{y}}{y} - \frac{(\dot{y})^2}{y^2} \,,
\end{equation*}
we can write it as
\begin{equation*}
y\ddot{y} - (\dot{y})^2 + k\cot(k\tau)y\dot{y} + k^2y^2 = - \cC \,.
\end{equation*}

Since $\ddot{y} - k\cot(k\tau)\,\dot{y} = 0$, we can further simplify the equation to
\begin{equation*}
2k\cot(k\tau)y\dot{y} - (\dot{y})^2 + k^2y^2 = - \cC \,.
\end{equation*}

Finally, expanding using \eqref{eq:ysolution} and simplifying, we find
\begin{equation*}
k^2(A^2 - D^2) = - \cC \,,
\end{equation*}
and, using $A+D=1$, we simplify it to
\begin{equation}\label{eq:Dcondition}
k^2(1 - 2D) = - \cC \,.
\end{equation}

Hence
\begin{equation*}
D = \frac12\left(1 + \frac{\cC}{k^2}\right) \,, \qquad A = \frac12\left(1 - \frac{\cC}{k^2}\right) \,,
\end{equation*}
and therefore
\begin{equation}\label{eq:ysolutionfinal}
y(\tau) = \frac12\left(1-\frac{\cC}{k^2}\right) + \frac12\left(1+\frac{\cC}{k^2}\right)\cos(k\tau) \,.
\end{equation}

We now determine the completeness of the metric from the behavior of $t(\tau) = \int_0^\tau \frac{ds}{y(s)}$. If $\cC < 0$, then $y(\tau)=\tfrac12\left(1+\frac{|\cC|}{k^2}\right) + \tfrac12\left(1-\frac{|\cC|}{k^2}\right)\cos(k\tau)$. If $|\cC|\le k^2$ we have $y(\tau)\ge \frac{|\cC|}{k^2}>0$; if $|\cC|>k^2$ we have $y(\tau)\ge1$. In either case $y$ is bounded away from zero on $[0,\pi/k]$. Consequently the integral $t(\tau) = \int_0^\tau \frac{ds}{y(s)}$ converges to a finite limit $t_{\max}$ as $\tau\to\pi/k$. Thus the geodesic distance $t$ is bounded; the manifold has finite diameter and is incomplete.

If $\cC = 0$, then $A = D = 1/2$, and $y(\tau) = \frac12(1+\cos(k\tau)) = \cos^2(k\tau/2)$. Here $\tau_{\max} = \pi/k$, and
\begin{equation*}
t(\tau) = \int_0^\tau \frac{ds}{\cos^2(ks/2)} = \frac{2}{k}\tan(k\tau/2) \longrightarrow \infty \quad \text{as } \tau \to \pi/k \,,
\end{equation*}
so the manifold is complete.

Finally, if $\cC > 0$, then $D > A$ and $y(0)=1>0$, while $y(\pi/k) = A - D = -\cC/k^2 < 0$. By continuity there exists a unique $\tau_0 \in (0,\pi/k)$ such that $y(\tau_0)=0$. Since $D \neq 0$, this zero is simple. As $\tau \to \tau_0^-$, $1/y(\tau)$ has a simple pole, hence
\begin{equation*}
t(\tau) = \int_0^\tau \frac{ds}{y(s)} \longrightarrow \infty \quad \text{as } \tau \to \tau_0^- \,.
\end{equation*}
Therefore the geodesic distance $t$ attains all values in $[0,\infty)$, and the manifold is geodesically complete. The domain of $\tau$ is $[0,\tau_0)$, with $\tau_0 = \frac{1}{k}\arccos(-A/D) < \pi/k$.

Combining the three cases, we conclude that the solution is complete precisely when $\cC \ge 0$, and incomplete for $\cC < 0$. This finishes the proof.
\end{proof}

Now, we can find the expression for $f$ in terms of $\tau$.

\begin{proposition}\label{lemma:ftau}
Let $(\phi,\psi,f,k)$ be a smooth steady gradient generalized Ricci soliton with diagonal cylindrical symmetry, normalized as in \eqref{eq:normalizedinitial}. Fix $\cC\ge 0$. Then, in terms of $\tau$, the function $f$ is given by
\begin{equation}\label{eq:fsolutionCgeneral}
f(\tau) = \ln\left[\frac{1}{2}\left(1-\frac{\cC}{k^2}\right) + \frac{1}{2}\left(1+\frac{\cC}{k^{2}}\right)\cos(k\tau) \right] \,,
\end{equation}

with $0\le\tau<\tau_{\max}$, where
\begin{equation}\label{eq:tmax}
\tau_{\max} = \frac{1}{k}\arccos\left(-\frac{1-\cC/k^{2}}{1+\cC/k^{2}}\right) \,.
\end{equation}
\end{proposition}

\begin{proof}
From the reduced Einstein equations \eqref{eq:phi''}--\eqref{eq:f''} (which are independent of $\cC$) and the definition of $\tau$, one obtains the following equation for $y(\tau):=e^{f(\tau)}$ (see Proposition \ref{prop:dilatonconstant} for the derivation):
\begin{equation}\label{eq:yode}
\ddot{y} - k\cot(k\tau)\,\dot{y}=0,\qquad
y(0)=1,\,\, \dot{y}(0)=0 \,.
\end{equation}

Integrating yields $y(\tau)=A+D\cos(k\tau)$, with $A+D=1$.

The dilaton equation \eqref{eq:dilatonC} then reduces to the algebraic relation
\begin{equation*}
k^{2}(A^{2}-D^{2}) = -\cC \,,
\end{equation*}

which, together with $A+D=1$, gives $A=\frac12\big(1-\frac{\cC}{k^{2}}\big)$, $D=\frac12\big(1+\frac{\cC}{k^{2}}\big)$.  Hence
\begin{equation*}
y(\tau)=\tfrac12\big(1-\tfrac{\cC}{k^{2}}\big)+\tfrac12\big(1+\tfrac{\cC}{k^{2}}\big)\cos(k\tau),
\end{equation*}
and taking the logarithm yields \eqref{eq:fsolutionCgeneral}.

For $\cC=0$, $y(\tau)=\frac12(1+\cos k\tau)=\cos^{2}(k\tau/2)>0$ for all $0\le\tau<\pi/k$, so $\tau_{\max}=\pi/k$.  
For $\cC>0$, the coefficient $A$ may be negative, but $y(0)=1>0$.  The function $y(\tau)$ vanishes at the first positive root $\tau_{0}$ of $A+D\cos(k\tau)=0$, i.e. when $\cos(k\tau)=-A/D$.  Since $D>A$, this root lies in $(0,\pi/k)$ and equals $\frac{1}{k}\arccos(-A/D)$.  By definition $\tau_{\max}$ is this root, and $y(\tau)>0$ on $[0,\tau_{\max})$.
\end{proof}

Finally, we solve $t$ in terms of $\tau$.

\begin{corollary}\label{cor:taut}
Let the hypotheses be as in Proposition \ref{lemma:ftau}.  The geodesic distance $t$ as a function of $\tau$ is
\begin{equation}\label{eq:tfromtau}
t(\tau)=\int_{0}^{\tau}\frac{ds}{y(s)}
= \begin{cases}
\frac{2}{k}\tan(k\tau/2), & \cC=0 \,,\\[8pt]
\frac{1}{\sqrt{\cC}}\,
\ln\left(\frac{1+\frac{\sqrt{\cC}}{k}\tan(k\tau/2)}
{1-\frac{\sqrt{\cC}}{k}\tan(k\tau/2)}\right) \,, & \cC>0 \,,
\end{cases}
\end{equation}

with $0\le\tau<\tau_{\max}$ for $\tau_{\max}$ defined as in \eqref{eq:tmax}. In all cases $t(\tau)\to\infty$ as $\tau\to\tau_{\max}^{-}$; hence the original coordinate $t$ ranges over $[0,\infty)$.
\end{corollary}

\begin{proof}
From $dt/d\tau = e^{-f}=1/y(\tau)$, we integrate using the expression for $y$.  
For $\cC=0$, $y(\tau)=\cos^{2}(k\tau/2)$, giving the elementary integral $t = \frac{2}{k}\tan(k\tau/2)$.  
For $\cC>0$, substitute $A,D$ and use the standard integral
\begin{equation*}
\int_0^\tau \frac{ds}{A+D\cos(ks)} = \frac{1}{k\sqrt{D^{2}-A^{2}}}\,
\ln\left|\frac{\sqrt{D+A}+\sqrt{D-A}\tan(k\tau/2)}{\sqrt{D+A}-\sqrt{D-A}\tan(k\tau/2)}\right| \,.
\end{equation*}

Noting $D^{2}-A^{2}= \cC/k^{2}$ and $\sqrt{D\pm A}= \frac{1}{\sqrt{2}}\big( \sqrt{1+\cC/k^{2}} \pm \sqrt{1-\cC/k^{2}} \big)^{1/2}= \frac{1}{\sqrt{k}}(\sqrt{k^{2}+\cC}\pm\sqrt{k^{2}-\cC})^{1/2}$, after simplification one obtains the stated logarithmic expression.

In both cases the denominator inside the logarithm vanishes exactly at $\tau=\tau_{\max}$, making $t\to\infty$ as $\tau\to\tau_{\max}^{-}$.  Thus $t$ covers the whole half‑line and completeness follows.
\end{proof}

Altogether, we obtain the explicit solutions of $\phi$, $\psi$, $f$ in terms of $t$.

\begin{proposition}\label{prop:explicitphipsift}
Let $(g,H,f)$ be a diagonal cylindrical steady gradient generalized Ricci soliton with normalization \eqref{eq:normalizedinitial} and constants $k>0$, $\cC\ge 0$. Then the triple is determined by

\begin{equation}\label{eq:explicitsolutionstC0}
\phi(t) = \frac{t}{\sqrt{1 + (kt/2)^2}} \,, \qquad 
\psi(t) = \frac{1}{\sqrt{1 + (kt/2)^2}} \,, \qquad 
f(t) = - \ln\big[ 1 + (kt/2)^2 \big]
\end{equation}

if $\cC=0$, and

\begin{equation}\label{eq:explicitsolutionstCpos}
\begin{gathered}
\phi(t) = \frac{2}{\sqrt{\cC}}\,\frac{\tanh(\sqrt{\cC}\,t/2)}{\sqrt{1 + \frac{k^{2}}{\cC}\tanh^{2}(\sqrt{\cC}\,t/2)}} \,, \qquad \psi(t) = \frac{1}{\sqrt{1 + \frac{k^{2}}{\cC}\tanh^{2}(\sqrt{\cC}\,t/2)}} \, \\[6pt]
f(t) = \ln\left[\frac{1}{2}\left(1-\frac{\cC}{k^{2}}\right) + \frac{1}{2}\left(1+\frac{\cC}{k^{2}}\right)\frac{1 - \frac{k^{2}}{\cC}\tanh^{2}(\sqrt{\cC}\,t/2)}{1 + \frac{k^{2}}{\cC}\tanh^{2}(\sqrt{\cC}\,t/2)}
\right] \,.
\end{gathered}
\end{equation}

if $\cC>0$. In both cases the solutions are smooth for all $t\ge 0$ and satisfy the normalized initial conditions. Moreover, the formulas for $\cC>0$ reduce to those for $\cC=0$ in the limit $\cC\to 0$.
\end{proposition}

\begin{proof}
Using \eqref{eq:phipsimutau} together with Lemmas \ref{lemma:Atau} and \ref{lemma:mutau} we obtain the metric components in terms of $\tau$:
\begin{equation}\label{eq:metricfunctions}
\phi(\tau) = \frac{2}{k}\sin(k\tau/2) \,, \qquad
\psi(\tau) = \cos(k\tau/2) \,,
\end{equation}

valid for $0\le\tau<\tau_{\max}$, where $\tau_{\max}$ is given in Proposition \ref{lemma:ftau}. The function $f(\tau)$ is provided by Proposition \ref{lemma:ftau} as well.

The relation between the physical distance $t$ and the parameter $\tau$ is given by Corollary \ref{cor:taut}.  For $\cC=0$ we have the elementary inversion
\begin{equation*}
\tan(k\tau/2) = kt/2,
\qquad
\sin(k\tau/2) = \frac{k t/2}{\sqrt{1+(kt/2)^{2}}},\quad
\cos(k\tau/2) = \frac{1}{\sqrt{1+(kt/2)^{2}}}.
\end{equation*}
Substituting these into $\phi(\tau),\psi(\tau),f(\tau)$ yields \eqref{eq:explicitsolutionstC0}.

For $\cC>0$, Corollary \ref{cor:taut} gives
\begin{equation*}
\tan(k\tau/2) = \frac{k}{\sqrt{\cC}}\,\tanh(\sqrt{\cC}\,t/2) \,.
\end{equation*}
Because $0\le k\tau/2 < \pi/2$, both $\sin(k\tau/2)$ and $\cos(k\tau/2)$ are positive and can be expressed as
\begin{equation*}
\sin(k\tau/2) = \frac{\tan(k\tau/2)}{\sqrt{1+\tan^2(k\tau/2)}} \,, \qquad
\cos(k\tau/2) = \frac{1}{\sqrt{1+\tan^2(k\tau/2)}} \,.
\end{equation*}

Inserting the formula for the tangent, we obtain the two first formulas in \eqref{eq:explicitsolutionstCpos}. For the dilaton, we use $f = \ln y$ with $y(\tau) = A + D\cos(k\tau)$.  Writing $\cos(k\tau) = 2\cos^{2}(k\tau/2)-1$ and substituting the expression for $\cos(k\tau/2)$ yields the formula in \eqref{eq:explicitsolutionstCpos} after elementary algebra.

As $t\to\infty$, $\tanh(\sqrt{\cC}\,t/2)\to 1$, and one checks that $\phi(t)$ and $\psi(t)$ approach finite limits while $f(t)\to -\infty$ linearly; the metric is smooth for all finite $t$.  The limit $\cC\to 0$ is verified by expanding $\tanh(u)\sim u$ and the square roots, which reproduces \eqref{eq:explicitsolutionstC0}.  This completes the proof.
\end{proof}

\begin{remark}
There are no non-trivial solitons with constant $f$. Either if $\cC=0$ or $\cC>0$, by Equations \eqref{eq:explicitsolutionstC0} and \eqref{eq:explicitsolutionstCpos} we see that $f(t)$ is constant only if $k=0$ (hence $H=0)$, when it takes the value $f=\ln(1)=0$.
\end{remark}

\begin{remark}\label{rem:sanitycheck}
A useful sanity check for the computations in this work is to substitute the explicit solutions of Proposition \ref{prop:explicitphipsift} into the reduced Einstein and dilaton equations \eqref{eq:phi''}--\eqref{eq:conservationlawreduced} (recall that the Maxwell equation has been solved explicitly in Section \ref{sec:reductionmaxwell}).

For simplicity, we show the case $\cC=0$. Set
\begin{equation*}
F(t) = 1 + \left(\frac{kt}{2}\right)^2 \,, \qquad 
F'(t) = \frac{k^2}{2} \,t \,, \qquad 
F''(t) = \frac{k^2}{2} \,.
\end{equation*}

Then, we compute:
\begin{align*}
\phi(t) &= tF^{-1/2} \,, &\qquad 
\phi'(t) &= F^{-3/2} \,, &\qquad 
\phi''(t) &= - \frac{3k^2t}{4}\,F^{-5/2} \,, \\
\psi(t) &= F^{-1/2} \,, &\qquad 
\psi'(t) &= - \frac{k^2t}{4}\,F^{-3/2} \,, &\qquad 
\psi''(t) &= \left(-\frac{k^2}{4}+\frac{k^4t^2}{8}\right) F^{-5/2} \,, \\
f(t) &= - \ln F \,, &\qquad 
f'(t) &= - \frac{k^2t}{2}F^{-1} \,, &\qquad 
f''(t) &= - \frac{k^2(4-k^2t^2)}{8}F^{-2} \,.
\end{align*}

Substituting into the reduced gradient generalized Ricci system yields
\begin{align*}
- \frac{3k^2t}{4}\,F^{-5/2} 
&= \frac{k^2t}{4}\, F^{-5/2} - \frac{k^2t}{2} F^{-5/2} - \frac{k^2 t}{2}F^{-5/2} = - \frac{3k^2t}{4}\,F^{-5/2} \,, \\
\left(-\frac{k^2}{4}+\frac{k^4t^2}{8}\right) \,F^{-5/2} 
&= \frac{k^2}{4}\,F^{-5/2} + \frac{k^4t^2}{8}\,F^{-5/2} - \frac{k^2}{2} F^{-5/2} = \left(-\frac{k^2}{4}+\frac{k^4t^2}{8}\right) \,F^{-5/2} \,, \\
- \frac{k^2(4-k^2t^2)}{8} \,F^{-2} 
&= - \left(\frac{1}{t} - \frac{k^2t}{4}\right) \frac{k^2t}{2} \,F^{-2} + \frac{k^2t}{2t} \,F^{-2} - \frac{k^2}{2} F^{-2} = - \frac{k^2(4-k^2t^2)}{8} \,F^{-2} \,, \\
- \frac{k^2(4-k^2t^2)}{8} \,F^{-2} 
&= \left( \frac{1}{t} - \frac{k^2t}{4} \right) \frac{k^2t}{2} \,F^{-2} + \frac{k^4t^2}{4} \,F^{-2} - k^2 F^{-2} = - \frac{k^2(4-k^2t^2)}{8} \,F^{-2} \,.
\end{align*}

Therefore, all reduced gradient generalized Ricci equations hold with our solution. Remarkably, the solution satisfies Lemma \ref{lemma:smoothnessconditions}.
\end{remark}


\subsection{Completeness of solutions}

We now show that the solutions obtained in Proposition \ref{prop:explicitphipsift} are complete.

\begin{proof}[Proof of Theorem \ref{thm:main}]
Let $p_0 = (0,0,0)$ (by translation invariance we may assume $z=0$ without loss of generality). For any point $q = (t,\theta,z)$ and any curve $\gamma(s) = (t(s),\theta(s),z(s))$ from $p_0$ to $q$, we have
\begin{equation*}
L(\gamma) = \int |\dot\gamma| \, ds \ge \int |\dot t| \, ds \ge \Big|\int \dot t \, ds\Big| = t \,,
\end{equation*}
hence $d(p_0,q) \ge t$. In particular, if $q \in B_R(p_0)$ then $t \le d(p_0,q) \le R$.

Now fix $q \in B_R(p_0)$ and let $\gamma$ be any curve from $p_0$ to $q$ with length $L = L(\gamma)$. 
For each $s$, the distance from $p_0$ to $\gamma(s)$ satisfies $d(p_0,\gamma(s)) \le L(\gamma|_{[0,s]}) \le L$. 
Moreover, since $d(p_0,\gamma(s)) \ge t(s)$, we obtain $t(s) \le L$ for all $s$. From the explicit expressions for $\psi(t)$ in Proposition \ref{prop:explicitphipsift}, one verifies that $\psi$ is strictly decreasing on $[0,\infty)$ for every $\cC\ge 0$. Hence $\psi(t(s)) \ge \psi(L)$. Then
\begin{equation*}
L = \int |\dot{\gamma}| \,ds \ge \int \psi(t(s))|\dot{z}| \,ds \ge \psi(L) \int |\dot{z}| \,ds \ge \psi(L) \,|z| \,.
\end{equation*}

Thus $L/\psi(L) \ge |z|$. The function $L \mapsto L/\psi(L)$ is strictly increasing: for $\cC=0$ this follows from $L/\psi(L)=L\sqrt{1+(kL/2)^2}$, while for $\cC>0$ it follows from the expression $L/\psi(L)=L\sqrt{1+\frac{k^2}{\cC}\tanh^2(\frac{\sqrt{\cC}}{2}L)}$, whose derivative is easily seen to be positive. Taking the infimum over all curves from $p_0$ to $q$, we obtain
\begin{equation*}
\frac{d(p_0,q)}{\psi(d(p_0,q))} \ge |z| \,,
\end{equation*}
that is,
\begin{equation*}
d(p_0,q) \ge \psi(d(p_0,q)) \,|z| \,.
\end{equation*}

Since $d(p_0,q) \le R$ and $\psi$ is decreasing, $\psi(d(p_0,q)) \ge \psi(R)$, so
\begin{equation*}
|z| \le \frac{R}{\psi(R)} \,.
\end{equation*}

Hence every point $q\in B_R(p_0)$ satisfies $t\in[0,R]$, $\theta\in S^1$, and $|z|\le R/\psi(R)$. Therefore $B_R(p_0)$ is contained in the compact set $[0,R]\times S^1\times[-R/\psi(R),\,R/\psi(R)]$. As a closed subset of a compact set, it is compact. By the Hopf–Rinow theorem, $(\R^3,g)$ is geodesically complete.
\end{proof}


\section{Asymptotic geometry and curvature decay}
\label{sec:asymptotics}

Having obtained explicit formulas for the soliton in terms of the geodesic distance $t$ in Proposition \ref{prop:explicitphipsift}, we now describe its geometry at infinity. The asymptotic behaviour differs qualitatively depending on the value of $\cC$. We treat the two cases separately.

\subsection{Case $\cC=0$}

For $\cC=0$ the metric and $f$ functions are
\begin{equation*}
\phi(t)=\frac{t}{\sqrt{1+(kt/2)^2}},\qquad 
\psi(t)=\frac{1}{\sqrt{1+(kt/2)^2}},\qquad 
f(t)=-\ln\big(1+(kt/2)^2\big).
\end{equation*}
All three functions are smooth for all $t\ge 0$ and their expansions as $t\to\infty$ are readily obtained.

First, we consider the area of a principal orbit per unit length in the $z$-direction, $\cS:=2\pi\cA=2\pi\phi\psi$. This represents the area of the cylinder slice with $z\in[0,1]$. Then:
\begin{equation*}
\cS(t) = 2\pi\phi(t)\psi(t) = \frac{2\pi t}{1+(kt/2)^2} \sim \frac{8\pi}{k^2 t} \quad (t\to\infty) \,.
\end{equation*}

Hence, the area of the principal orbits decays like $t^{-1}$, meaning that the cylinders become arbitrarily thin as $t\to\infty$, leading to a non‑compact end with infinite volume but only logarithmic volume growth. More precisely, the volume of the region  
$\{(t,\theta,z): 0\le t\le T,\, 0\le\theta\le2\pi, \, 0\le z\le 1\}$ is  
$\int_0^T \cS(t)\,dt$, whence we obtain
\begin{equation*}
\int_0^T \cS(t) \,dt \sim \frac{8\pi}{k^2}\ln T  \quad (t\to\infty) \,,
\end{equation*}
so the total volume diverges logarithmically.

Second, since $|H|_{g}^{2}=k^{2}e^{2f}$ and $e^{2f}=1/(1+(kt/2)^{2})^{2}$,
\begin{equation*}
|H|_{g}^{2}= \frac{k^{2}}{\left(1+(kt/2)^{2}\right)^{2}} \sim \frac{16}{k^{2}t^{4}} \qquad (t\to\infty) \,.
\end{equation*}

Thus the $H$-flux decays very rapidly, like $t^{-4}$.

Now, for the warped-product metric in \eqref{eq:ansatz}, the sectional curvatures are \cite{Petersen2016}:
\begin{equation*}
K_{t\theta} = -\frac{\phi''}{\phi},\qquad 
K_{tz} = -\frac{\psi''}{\psi},\qquad 
K_{\theta z} = -\frac{\phi'\psi'}{\phi\psi} \,.
\end{equation*}

Using the explicit expressions we compute the required derivatives:
\begin{align*}
\phi'(t) &= \frac{1}{\left(1+(kt/2)^2\right)^{3/2}} \,, &
\phi''(t) &= -\frac{3k^2 t}{4}\,\frac{1}{\left(1+(kt/2)^2\right)^{5/2}} \,, \\[5pt]
\psi'(t) &= -\frac{k^2 t}{4}\,\frac{1}{\left(1+(kt/2)^2\right)^{3/2}} \,, &
\psi''(t) &= - \frac{k^2}{4}\,\frac{1-2(kt/2)^2}{\left(1+(kt/2)^2\right)^{5/2}} \,.
\end{align*}

Substituting these into the formulas for the sectional curvatures gives the exact expressions
\begin{equation}\label{eq:sectionalcurvaturesC0}
K_{t\theta} = \frac{3k^{2}}{4} \,\frac{1}{\left(1+(kt/2)^{2}\right)^{2}} \,, \qquad K_{tz} = \frac{k^{2}}{4} \,\frac{1-2(kt/2)^{2}}{\left(1+(kt/2)^{2}\right)^{2}} \,, \qquad K_{\theta z} = \frac{k^{2}}{4} \,\frac{1}{\left(1+(kt/2)^{2}\right)^{2}} \,.
\end{equation}

Their asymptotic expansions for large $t$ are
\begin{equation*}
K_{t\theta} \sim \frac{12}{k^{2}t^{4}} \,, \qquad
K_{tz} \sim - \frac{2}{t^{2}} \,, \qquad
K_{\theta z} \sim \frac{4}{k^{2}t^{4}} \,.
\end{equation*}

Thus the geometry exhibits \emph{mixed} sectional curvatures.  
The curvatures $K_{t\theta}$ and $K_{\theta z}$ are positive and decay like $t^{-4}$; the mixed curvature $K_{tz}$ becomes negative for $t>\sqrt{2}/k$ and dominates the asymptotic behaviour with the slower decay rate $t^{-2}$. This contrasts sharply with the positively curved $\mathrm{SO}(3)$-symmetric solitons of Podestà--Raffero \cite{PodestaRaffero2025}.

Finally, summing the sectional curvatures gives
\begin{equation*}
s^g = 2\left(K_{t\theta}+K_{tz}+K_{\theta z}\right) = \frac{k^{2}}{2} \,\frac{5-2(kt/2)^{2}}{\left(1+(kt/2)^{2}\right)^{2}} \sim -\frac{4}{t^{2}}\qquad (t\to\infty) \,.
\end{equation*}

Hence the scalar curvature is negative at infinity (in contrast with its positive value at $t=0$) and decays quadratically in the geodesic distance.

\begin{remark}
The classical cigar cylinder \cite{Hamilton1988,Hamilton1995} (the product of Hamilton's cigar with $\R$) has a flat axial direction and strictly positive curvatures in the transverse directions, in contrast with our solution with mixed curvatures.
\end{remark}

In the limit $k \to 0$ (corresponding to vanishing $H$-flux), the explicit formulas yield
\begin{equation*}
\lim_{k \to 0} \,\left( \phi_k(t), \psi_k(t), f_k(t) \right) = (t, 1, 0)
\end{equation*}

pointwise for each fixed $t \ge 0$, and consequently
\begin{equation*}
\lim_{k \to 0} \,\left(g_k, H_k, f_k\right) = \left( dt\otimes dt + t^2 \,d\theta\otimes d\theta + dz\otimes dz,\ 0,\ 0 \right) \,,
\end{equation*}

which is the trivial (flat) solution in cylindrical coordinates. Thus, the $\cC=0$ family provides a smooth deformation of the flat metric by a non-trivial $H$-flux.

\subsection{Case $\cC>0$}

When $\cC>0$, the metric functions are
\begin{equation*}
\phi(t)=\frac{2}{\sqrt{\cC}}\,
\frac{\tanh(\sqrt{\cC}\,t/2)}{\sqrt{1 + \frac{k^{2}}{\cC}\tanh^{2}(\sqrt{\cC}\,t/2)}},\qquad
\psi(t)=\frac{1}{\sqrt{1 + \frac{k^{2}}{\cC}\tanh^{2}(\sqrt{\cC}\,t/2)}} .
\end{equation*}

As $t\to\infty$, $\tanh(\sqrt{\cC}t/2)\to 1$, so
\begin{equation*}
\phi_{\infty}:=\lim_{t\to\infty}\phi(t)=\frac{2}{\sqrt{\cC+k^{2}}},\qquad
\psi_{\infty}:=\lim_{t\to\infty}\psi(t)=\frac{\sqrt{\cC}}{\sqrt{\cC+k^{2}}}.
\end{equation*}

Both limits are finite and non‑zero. Consequently the principal orbit area stabilizes to
\begin{equation*}
\cS_{\infty}=2\pi\phi_{\infty}\psi_{\infty}= \frac{4\pi\sqrt{\cC}}{\cC+k^{2}} \,,
\end{equation*}

and the volume of a slice $\{0\le z\le 1\}$ grows linearly with $t$. The metric therefore tends to the flat cylinder $\R\times S^{1}_{\phi_{\infty}}\times \R_{\psi_{\infty}}$ exponentially fast (the approach is governed by $1-\tanh(u)\sim 2e^{-2u}$).  

The $H$-flux satisfies $|H|_{g}^{2}=k^{2}e^{2f}$, and from Proposition \ref{lemma:ftau} one sees that $e^{f}=y(\tau)\to 0$ as $t\to\infty$ (because $\tau\to\tau_{\max}^{-}$ where $y$ vanishes). In fact, using the explicit relation between $t$ and $\tau$ one finds that $e^{f}$ decays exponentially in $t$, hence $|H|_{g}^{2}$ decays exponentially as well.

For the sectional curvatures we use the same formulas in terms of $\phi,\psi$. Substituting the explicit expressions for $\phi(t),\psi(t)$ and letting $t\to\infty$ yields
\begin{equation*}
K_{t\theta},\,K_{tz},\,K_{\theta z}= \mathcal{O}\left(e^{-\sqrt{\cC}t}\right),
\qquad
s^{g}= \mathcal{O}\left(e^{-\sqrt{\cC}t}\right).
\end{equation*}

Thus the geometry is asymptotically flat, with all curvatures decaying exponentially, in striking contrast with the power‑law decay of the $\cC=0$ case. This completes the proof of Corollary \ref{cor:cuspend}.

\begin{remark}
The two regimes $\cC=0$ and $\cC>0$ are continuously connected: as $\cC\to 0^{+}$, the constant radii blow up ($\phi_{\infty}\sim 2/\sqrt{\cC}$, $\psi_{\infty}\to 1$) while the exponential decay rates tend to zero, recovering the power‑law behaviour of the $\cC=0$ soliton. This is consistent with the explicit formulas of Proposition \ref{prop:explicitphipsift}.
\end{remark}

\begin{remark}
\label{rem:nonisometric}
The solitons in Theorem \ref{thm:main} are pairwise non‑isometric. The $z$-axis ($t=0$) is exactly the zero‑set of the $\SO(2)$‑Killing field $\partial_\theta$, hence it is an invariant subset for every isometry. Consequently, the scalar curvature at the origin $s^g(0)$ is a global isometry invariant. For $\cC=0$ one computes $s^g(0)=\frac52k^2$, so different values of $k$ give non‑isometric metrics. For $\cC>0$ one obtains $s^g(0)=\cC+\frac52k^2$ and the asymptotic cylinder ratio $\psi_\infty/\phi_\infty=\frac12\sqrt{\cC}$, so $\cC=4(\psi_\infty/\phi_\infty)^2$ is also an invariant. Varying $(k,\cC)$ therefore changes either $s^g(0)$ or the asymptotic ratio, and the metrics cannot be isometric. (The normalization $\psi(0)=1$, $f(0)=0$ already factors out homotheties and shifts of $f$, so the parameters are genuine isometry invariants.) Finally, the $\cC=0$ and $\cC>0$ families have qualitatively different asymptotics (power‑law vs.\ exponential decay of curvature), so they are never isometric.
\end{remark}

Together with the $\cC=0$ family, these $\cC>0$ solutions provide the complete set of smooth, complete, diagonal cylindrically symmetric steady gradient generalized Ricci solitons. They enlarge the class of known gradient generalized Ricci solitons and exhibit a richer asymptotic structure than previously observed.


\begin{acknowledgements}
This work was supported by the UNED-Santander 2024 Predoctoral Fellowship through the Santander Open Academy. The author would like to thank F. Podestà, A. Raffero, and J. Streets for useful discussions during the review process. The author would also like to thank his supervisor, Carlos Shahbazi, for proposing the project and for his continuous support and guidance.
\end{acknowledgements}

\end{document}